\documentclass[leqno,12pt]{amsart}
\usepackage{latexsym}
\usepackage{mathrsfs}
\usepackage{amsmath}
\usepackage{amssymb}
\usepackage[bookmarks]{hyperref}
\usepackage{pstricks,pst-plot}

\font\tenscr=rsfs10 
\font\sevenscr=rsfs7 
\font\fivescr=rsfs5 
\skewchar\tenscr='177 \skewchar\sevenscr='177 \skewchar\fivescr='177
\newfam\scrfam \textfont\scrfam=\tenscr \scriptfont\scrfam=\sevenscr
\scriptscriptfont\scrfam=\fivescr
\def\scr{\fam\scrfam}

\font\tenscr=rsfs10 
\font\sevenscr=rsfs7 
\font\fivescr=rsfs5 
\skewchar\tenscr='177 \skewchar\sevenscr='177 \skewchar\fivescr='177
\newfam\scrfam \textfont\scrfam=\tenscr \scriptfont\scrfam=\sevenscr
\scriptscriptfont\scrfam=\fivescr
\def\scr{\fam\scrfam}

\newtheorem{theorem}{Theorem}[section]

\newtheorem{corollary}[theorem]{Corollory}
\newtheorem{lemma}[theorem]{Lemma}
\newtheorem{question}[theorem]{Question}

\newtheorem{remark}[theorem]{Remark}

\newtheorem{definition} [theorem]{Definition}

\newcommand{\bthm}{\begin{theorem}}
\newcommand{\ethm}{\end{theorem}}
\newcommand{\blem}{\begin{lemma}}
\newcommand{\elem}{\end{lemma}}
\newcommand{\bcor}{\begin{corollary}}
\newcommand{\ecor}{\end{corollary}}
\newcommand{\bprop}{\begin{proposition}}
\newcommand{\eprop}{\end{proposition}}
\newcommand{\bdefn}{\begin{definition}}
\newcommand{\edefn}{\end{definition}}
\newcommand{\bpf}{\begin{proof}}
\newcommand{\epf}{\end{proof}}

\newcommand{\bi}{\begin{itemize}}
\newcommand{\ei}{\end{itemize}}

\newcommand{\bc}{\begin{cases}}
\newcommand{\ec}{\end{cases}}

\newcommand{\ba}{\begin{array}}
\newcommand{\ea}{\end{array}}

\newcommand{\be}{\begin{equation}}
\newcommand{\ee}{\end{equation}}

\newcommand{\bea}{\begin{eqnarray}}
\newcommand{\eea}{\end{eqnarray}}

\newcommand{\beaa}{\begin{eqnarray*}}
\newcommand{\eeaa}{\end{eqnarray*}}

\newcommand{\beastar}{\begin{eqnarray*}}
\newcommand{\eeastar}{\end{eqnarray*}}

\begin{document}

\def\tA{\tilde A}
\def\tX{\tilde X}
\def\tf{\tilde f}
\def\tpi{\tilde \pi}
\def\th{\tilde h}
\def\ta{\tilde \alpha}

\def \vep {\varepsilon}

\def \cd {, \ldots ,}
\def \lf{\| f \|}
\def \bs{\setminus}
\def \ep{\varepsilon}
\def \sig{\Sigma}
\def \si {\sigma}
\def \gam {\gamma}
\def \cinf {C^\infty}
\def \cid {C^\infty (\partial D)}

\def\h#1{\widehat {#1}}
\def\hh#1{\widehat {#1} \bs  {#1}}
\def\hk#1#2{{\widehat {#2}}^{#1}}
\def\hhk#1#2{{\widehat {#2}}^{#1} \bs {#2}}

\def\hr#1{h_r({#1})}
\def\hhr#1{h_r({#1}) \bs {#1}}
\def\hrk#1#2{h_r^{#1}({#2})}
\def\hhrk#1#2{h_r^{#1}({#2}) \bs {#2}}

\def\kh#1#2{{}^{#1}{\widehat {#2}}}
\def\khr#1#2{{}^{#1}h_r({#2})}
\def\kshr#1#2{{}_{#1}h_r({#2})}

\def\<{\langle}
\def\>{\rangle}

\def\spshell{\{ z\in \CN: 1<\|z\| < \rho\}}

\def \C {\mathbb C}
\def \CN {\mathbb C^N}
\def\R {\mathbb R}

\def \Z {\mathbb Z}

\def \bbr {\mathbb R}
\def \bbrs {\mathbb R^2}
\def \bbrn {\mathbb R^n}

\def \cc {\mathscr C}
\def \kk {\mathscr K}
\def \mm {\mathfrak M}

\def \od {\overline D}
\def \ol {\overline L}
\def \oj {\overline J}

\def\sC{{\scr C}}
\def\sF{{\scr F}}
\def\sG{{\scr G}}

\def \cn {c_1, \ldots , c_n}
\def \zx { z \in \widehat X }
\def \zxx {\{ z \in \widehat X \} }
\def \siv {\sum^\infty _{j=1}}
\def \sivv {\sum^\infty _{j=n+1}}
\def \smv {\sum^m_{j=n+1}}
\def \snv {\sum^n_{j=j_0+1}}
\def \bjkz {B_{jk} (z_0)}
\def \epp {\varepsilon^2_n/2}
\def \epn {\varepsilon_{n+1}}
\def \nn { {n+1} }
\def \nnn {{2N}}
\def \ff {F_{z_{0}}}
\def \fk {f_{k_{1}}}
\def \fkv {f_{k_{v}}}
\def \tf {\tilde{f}}

\def \ma {\mathfrak{M}_A}  
\def \mb {\mathfrak{M}_B}
\def \muc {\mathfrak{M}_{\mathscr U}}
\def \uc {\mathscr U}
\def \xy {(x,y)}
\def \fg {f \otimes g}
\def \vp {\varphi}
\def \pl {\partial L}
\def \ad {A(D)}
\def \aid {A^\infty (D)}
\def \ot {\otimes}
\def \rbl {R_b (L)}
\def\Int {{\rm Int}}

\def\rowonly#1#2{#1_1,\ldots,#1_#2}
\def\row#1#2{(#1,\ldots,#1_#2)}
\def\diam{\mathop{\rm diam}\nolimits}

\def\pB{\partial B}
\def\oB{\overline B}

\def\endhat{\widehat{\phantom j}}
\def\endhatk{\widehat{\phantom j}{\phantom |}^k}

\def\pji{p_j^{-1}}
\def\oh{\overline H}

\subjclass[2000]{32E20, 46J10, 46J15}
\title[Polynomial hulls that contain no analytic discs]{Spaces with polynomial hulls\\ that contain no analytic discs}
\author{Alexander J. Izzo}
\address{Department of Mathematics and Statistics, Bowling Green State University, Bowling Green, OH 43403}
\email{aizzo@bgsu.edu}

\begin{abstract}
\vskip 24pt
Extensions of the notions of polynomial and \hbox{rational} hull are introduced.  Using these notions, a generalization of a result of Duval and Levenberg on polynomial hulls containing no analytic discs is presented.  As a consequence it is shown that there exists a Cantor set in $\C^3$ with a nontrivial polynomial hull that contains no analytic discs.  Using this Cantor set, it is shown that there exist arcs and simple closed curves in $\C^4$ with nontrivial polynomial hulls that contain no analytic discs.  This answers a question raised by Bercovici in 2014 and can be regarded as a partial answer to a question raised by Wermer over 60 years ago.  More generally, it is shown that
every uncountable, compact subspace of a Euclidean space can be embedded as a subspace $X$ of $\CN$, for some $N$, in such a way as to have a nontrivial polynomial hull that contains no analytic discs.
In the case when the topological dimension of the space is at most one, $X$ can be chosen so as to have the stronger property that $P(X)$ has a dense set of invertible elements. 
\end{abstract}

\maketitle

\vskip -3.52 true in
\centerline{\footnotesize\it Dedicated to to the memory of Donald Sarason} 
\vskip 3.52 truein

 \section{Introduction}
It was once conjectured that whenever the polynomial hull $\h X$ of a compact set $X$ in $\CN$ is strictly larger than $X$, the complementary set $\hh X$ must contain an analytic disc. This conjecture was disproved by Gabriel Stolzenberg \cite{Stol1}.  However, when $X$ is a smooth one-dimensional manifold, the set $\hh X$, if nonempty, is a one-dimensional analytic variety as was also shown by Stolzenberg \cite{Stol2} (strengthening earlier results of several mathematicians).
In contrast, recent work of the author, H\aa kan Samuelsson Kalm, and Erlend Forn{\ae}ss Wold \cite{ISW} and of the author and Lee Stout \cite{IS} shows that every smooth manifold of dimension strictly greater than one smoothly embeds in some $\CN$ as a subspace $X$ such that $\hh X$ is nonempty but contains no analytic discs.

In response to a talk on the above results given by the author, Hari Bercovici raised the question of whether a {\it nonsmooth} one-dimensional manifold can have polynomial hull containing no analytic discs. This question was the motivation for the present paper and will be answered affirmatively.  In fact, it will be shown, in Theorem~\ref{maintheorem} below, that every uncountable, compact subspace of a Euclidean space can be embedded in some $\CN$ so as to have polynomial hull containing no analytic discs. It is pleasure to thank Bercovici for his question which had a very stimulating effect on the author's research.

A similar question was in fact raised by John Wermer \cite{Wermer1954} more than 60 years ago.  Wermer observed that for $\phi_1$ and $\phi_2$ continuous complex-valued functions separating points on the unit circle $\partial D$, a necessary condition for the polynomials in the complex coordinate functions to be uniformly dense in the continuous functions on the simple closed curve $J=\bigl\{\bigl(\phi_1(z),\phi_2(z)\bigr):z\in \partial D\bigr\}$ is that $J$ not bound a piece of an analytic variety, and he proved that the condition is also sufficient when either $\phi_1(z)=z$ or $\phi_1(z)=z^2$.  He conjectured that the condition is always sufficient.  While this conjecture is still open,  Theorem~\ref{arc} below establishes that the conjecture becomes false if the pair of continuous functions $(\phi_1,\phi_2)$ is replaced by a quadruplet $(\phi_1\cd \phi_4)$.

In stating our results, it will be convenient to say that a hull $X_H$ of a set $X \subset \CN$ is {\it nontrivial} if the set $X_H\setminus X$ is nonempty.

\begin {theorem} \label{maintheorem} 
Let $K$ be an arbitrary uncountable, compact subspace of $\R^n$. Then there exists a subspace $X$ of $\C^{n+4}$ homeomorphic to $K$ such that the polynomial hull $\h X$ of $X$ is nontrivial but contains no analytic discs.
\end {theorem}

Since the compact subspaces of Euclidean spaces are precisely the compact metrizable spaces of finite topological dimension \cite[Theorem~V.2]{HW}, the above theorem says, in particular, that every uncountable, compact metrizable space of finite topological dimension embeds in some $\CN$ so as to have nontrivial polynomial hull containing no analytic discs.  More precisely we have the following corollary.  (By topological dimension, we mean the usual Lebesgue covering dimension.)

\begin {corollary} \label{maincorollary} 
If $K$ is an uncountable, compact metrizable space of topological dimension $m$, then there exists a subspace $X$ of $\,\C^{2m+5}$ homeomorphic to $K$ such that the polynomial hull $\h X$ of $X$ is nontrivial but contains no analytic discs.
\end {corollary}

The key to proving Theorem~\ref{maintheorem} is to obtain a {\it Cantor set\/} whose polynomial hull is nontrivial but contains no analytic discs.  Once this is done, all other desired spaces can be handled by a devise used repeatedly in the author's joint work with Samuelsson Kalm, Wold, and Stout \cite{ISW, IS}. The first example of a Cantor set with nontrivial polynomial hull was constructed by Walter Rudin \cite{Rudin-Cantor} using a modification of an argument of Wermer  \cite{Wermer-Cantor} who produced the first example of an {\it arc\/} with nontrivial polynomial hull. 
We will prove the following.

\bthm \label{Cantorset} 
There exists a Cantor set in $\C^3$ whose polynomial hull is nontrivial but contains no analytic discs.
\ethm

An alternative construction of a Cantor set with nontrivial polynomial hull containing no analytic discs is given in the paper \cite{I-L} of the present author and Norman Levenberg.  In addition to being more direct than the construction in \cite{I-L}, the construction to be presented here has the advantage of being very flexible and making possible the introduction of various additional properties.  One of these concerns density of invertible elements.
Garth~Dales and Joel~Feinstein \cite{DalesF} constructed a compact set $X$ in $\C^2$ with nontrivial polynomial hull and such that the uniform algebra $P(X)$ has a dense set of invertible elements.  
Additional examples of this phenomenon were later given by the present author \cite{Izzo}.
As noted in \cite{DalesF} the condition that $P(X)$ has a dense set of invertible elements is strictly stronger than the condition that $\h X$ contains no analytic disc.  
Following Dales and Feinstein, 
we will say that a uniform algebra $A$ has dense invertibles if the invertible elements of $A$ are dense in $A$.
We will prove the following result which contains Theorem~\ref{Cantorset}

\bthm\label{Cantorset-dense-invert}
There exists a Cantor set $X$ in $\C^3$ with nontrivial polynomial hull such that $P(X)$ has dense invertibles.
\ethm

Using this Cantor set, we will show that 
every uncountable, compact metrizable space of topological dimension at most 1 can be embedded in some $\C^N$ so as to have nontrivial polynomial hull and a polynomial algebra with dense invertibles.

\begin {theorem} \label {one-d-gen} 
Let $K$ be an uncountable, compact subspace of $\R^n$ of topological dimension at most $1$. Then there exists a subspace $X$ of $\C^{n+4}$ homeomorphic to $K$ such that the polynomial hull $\h X$ of $X$ is nontrivial and the uniform algebra $P(X)$ has dense invertibles.
\end {theorem}

\begin {corollary} \label {one-d-special} 
If $K$ is an uncountable, compact metrizable space of topological dimension at most $1$, then there exists a subspace $X$ of $\,\C^{7}$ homeomorphic to $K$ such that the polynomial hull $\h X$ of $X$ is nontrivial and the uniform algebra $P(X)$ has dense invertibles.
\end {corollary}

By an {\it arc} we shall mean a space homeomorphic to the closed unit interval and by a {\it simple closed curve}, a space homeomorphic to the unit circle.  Thus the arcs and simple closed curves are the compact, connected, one-dimensional manifolds.  It follows immediately from 
Theorem~\ref{maintheorem} that there is an arc in $\C^5$ with nontrivial polynomial hull containing no analytic discs and a simple closed curve in $\C^6$ with the same property.  However, by applying Theorem~\ref{Cantorset} more directly, we will obtain examples in $\C^4$.

\begin {theorem} \label {arc} 
There exists an arc $X$ in $\C^4$ whose polynomial hull is nontrivial but contains no analytic discs.  The same statement holds with \lq \lq arc\rq\rq\ replaced by \lq \lq simple closed curve\rq \rq.  Furthermore, the arc or curve $X$ can be chosen so that $P(X)$ has dense invertibles.
\end {theorem}

Given that a nontrivial hull need not in general contain an analytic disc, it is natural to ask whether weaker semblences of analyticity must be present such as nontrivial Gleason parts or nonzero bounded point derivations.  It is shown in the paper \cite{CGI} of Brian~Cole, Swarup~Ghosh, and the present author that even these weaker forms of analyticity need not be present in nontrivial hulls.  On the other hand, as shown in the paper \cite{Izzo} of the present author, there are also examples of hulls without analytic discs that have large Gleason parts and an abundance of nonzero bounded point derivations.  (The first result concerning Gleason parts in hulls without analytic discs is due to Richard Basener \cite{Basener}.)  By combining the methods of the present paper with those of \cite{CGI} and \cite{Izzo}, these results of \cite{CGI} and \cite{Izzo} will be sharpened by showing that the set in question can be taken to be a Cantor set.  

The construction of Cantor sets with nontrivial hulls without analytic discs to be given here involves an extension of the notions of polynomial and rational hull to be presented in Section~\ref{definitions}.  
We will define, for a compact set $X$ in $\CN$ and each $k=1,\ldots, N$, sets $\hk kX$ and $\hrk kX$, which we call the $k$-polynomial and $k$-rational hulls of $X$, respectively.  It will be immediate from the definitions that 
$$\h X=\hk 1X\supset \hr X=\hrk 1X\supset \hk 2X\supset \hrk 2X 
\supset\cdots\supset \hk NX \supset \hrk NX=X.$$
Using these new hulls, 
a generalization of the construction of Julien~Duval and Levenberg \cite{DuvalL} of hulls without analytic discs will be given that brings out the fundamental principle behind the construction and
yields a very flexible method of obtaining hulls without analytic discs (Theorem~\ref{Duval}).  
In particular, it will be shown that every compact set $X$ having nontrivial
2-polynomial hull contains a subset having polynomial hull without analytic discs (Corollary~\ref{main-cor}).  This was the key observation that led the author to the results of the present paper.
The existence, for each $N\geq 2$, of a Cantor set in $\C^N$ whose $(N-1)$-rational hull has interior will be proven (Theorem~\ref{CantorLhull}).  The existence of the desired Cantor set with polynomial hull without analytic discs then follows.  Another consequence is the existence of a totally disconnected, perfect subset of $\CN$ that intersects every analytic subvariety of $\CN$ of positive dimension (Theorem~\ref{perfectset}).

In response to a question raised by Rudin~\cite{Birtel},
Vitushkin~\cite{Vit} constructed a Cantor set in $\C^2$ whose polynomial hull has interior.  Later Henkin~\cite{Henkin} constructed a Cantor set in $\C^2$ whose \emph{rational} hull has interior.  
Theorem~\ref{CantorLhull} generalizes
Henkin's result to $\C^N$, $N>2$, in a particularly strong form.  It was the problem of suitably generalizing Henkin's result that originally motivated the author to consider the new hulls introduced in this paper.

By moving to an abstract uniform algebra context, one can consider removing the restriction in Theorem~\ref{maintheorem} that $K$ be a subspace of a Euclidean space.
In the case of compact Hausdorff spaces that contain a Cantor set, one obtains the following result.

\begin {theorem} \label{generalspace} 
Let $K$ be a compact Hausdorff space that contains a Cantor set as a subspace.  Then there exists a uniform algebra $A$ on $K$ such that $\ma \bs K$ is nonempty but $\ma$ contains no analytic discs.  {\rm(}Here $\ma$ denotes the maximal ideal space of $A$.{\rm)}
\end {theorem}

This theorem applies to every uncountable, compact metrizable space since these spaces always contain a Cantor set.  When one moves beyond metrizable spaces, there are uncountable, compact Hausdorff spaces on which there are no nontrivial uniform algebras (i.e., uniform algebras other than the algebra of all continuous complex-valued functions on the given space), and obviously no theorem along the lines of Theorem~\ref{generalspace} can apply to such spaces.  Indeed a theorem of Walter Rudin \cite{Rudin-scat} asserts that there are no nontrivial uniform algebras on a scattered space.  (A space $X$ is said to be {\it scattered} if it contains no nonempty perfect subsets, or equivalently, if every nonempty subset of $X$ has an isolated point.)
Note that this implies, in particular, that there are never nontrivial uniform algebras on countable spaces.  There are nonscattered, compact Hausdorff spaces that contain no Cantor set, and on some of these spaces (for instance the Stone-\v Cech compactification of the positive integers) there do exist nontrivial uniform algebras, and on some others (for instance the double arrow space of Aleksandrov and Urysohn) there are no nontrivial uniform algebras.  We refer the reader to the paper \cite{Kunen} of Kenneth Kunen and the references
therein.  
These remarks lead to the following two questions, which in the case of nonmetrizable spaces, seem to be open.
While these questions are clearly related to the subject of the present paper, they are of a quite different character and will not be addressed here.

\begin{question}
If $X$ is a compact Hausdorff space on which there exists a nontrivial uniform algebra, must there exist a uniform algebra on $X$ whose maximal ideal space is strictly larger than $X$?
\end{question}

\begin{question}
If $X$ is a compact Hausdorff space on which there exists  a uniform algebra whose maximal ideal space is strictly larger than $X$, must there exist a uniform algebra on $X$ whose maximal ideal space is strictly larger than $X$ but contains no analytic discs?
\end{question}

In the next section
we make explicit some standard definitions and notations already used above.  
In Section~\ref{definitions} we introduce the extensions of the notions of polynomial and rational hull upon which our constructions rely.  Some useful lemmas are proved in Section~\ref{lemmas}.  The generalization of the construction of Duval and Levenberg and some variations are given in Section~\ref{DuvalL-section}.  The Cantor set in $\CN$ whose 
$(N-1)$-rational hull contains interior 
and the totally disconnected perfect set that intersects every analytic variety are constructed in Section~\ref{Cantor-set-section}.  The Cantor sets with nontrivial polynomial hulls containing no analytic discs are given in Section~\ref{Cantor-no-discs}.  
Finally, the results about arcs, curves, and general spaces with hulls without analytic discs are proved in Section~\ref{spaces-no-discs}.

The author would like to thank Lee Stout for helpful correspondence related to this paper.

It is with a mixture of joy and sorrow that I dedicate this paper to the memory of my doctoral advisor, Donald Sarason.  Sorrow, of course, that he is no longer with us; joy that
I had the privilege of studying under him and benefiting from his guidance.  Sarason had an amazing ability to impart his wisdom and inspire his students while encouraging them to find their own way.  Without his guidance, my path would surely have been very different.

%

\section{Preliminaries}~\label{prelim}
For
$X$ a compact Hausdorff space, we denote by $C(X)$ the algebra of all continuous complex-valued functions on $X$ with the supremum norm
$ \|f\|_{X} = \sup\{ |f(x)| : x \in X \}$.  A \emph{uniform algebra} on $X$ is a closed subalgebra of $C(X)$ that contains the constant functions and separates
the points of $X$.  

For a compact set $X$ in $\CN$, the \emph{polynomial hull} $\h X$ of $X$ is defined by
$$\h X=\{z\in\CN:|p(z)|\leq \max_{x\in X}|p(x)|\
\mbox{\rm{for\ all\ polynomials}}\ p\},$$
and the
\emph{rational hull} $\hr X$ of $X$ is defined by
$$\hr X = \{z\in\C^N: p(z)\in p(X)\ 
\mbox{\rm{for\ all\ polynomials}}\ p
\}.$$
An equivalent formulation of the definition of $\hr X$ is that $\hr X$ consists precisely of those points $z\in \C^N$ such that every polynomial that vanishes at $z$ also has a zero on $X$.

We denote by 
$P(X)$ the uniform closure on $X\subset\CN$ of the polynomials in the complex coordinate functions $z_1,\ldots, z_N$, and we denote by $R(X)$ the uniform closure of the rational functions  holomorphic on (a neighborhood of) $X$. 
Both $P(X)$ and $R(X)$ are uniform algebras, and
it is well known that the maximal ideal space of $P(X)$ can be naturally identified with $\h X$, and the maximal ideal space of $R(X)$ can be naturally identified with $\hr X$.

We denote the open unit disc in the complex plane by $D$.  By an \emph{analytic disc} in $\CN$, we mean an injective holomorphic map $\sigma: D\rightarrow\CN$.
By the statement that a subset $S$ of $\CN$ contains no analytic discs, we mean that there is no analytic disc in $\CN$ whose image is contained in $S$.
An \emph{analytic disc} in the maximal ideal space $\ma$ of a uniform algebra $A$ is, by definition, an injective map $\sigma:D\rightarrow \ma$ such that the function $f\circ \sigma$ is analytic on $D$ for every $f$ in 
$A$. 

Let $A$ be a uniform algebra on a compact space $X$.
The \emph{Gleason parts} for the uniform algebra $A$ are the equivalence classes in the maximal ideal space of $A$ under the equivalence relation $\varphi\sim\psi$ if $\|\varphi-\psi\|<2$ in the norm on the dual space $A^*$.  (That this really is an equivalence relation is well-known but {\it not\/} obvious!)
We say that a Gleason part is \emph{nontrivial} if it contains more than one point.

For $\phi$ a multiplicative linear functional on $A$, a \emph{point derivation} on $A$ at $\phi$ is a linear functional $\psi$ on $A$ satisfying the identity 
$$\phantom{\hbox{for all\ } f,g\in A.} \psi(fg)=\psi(f)\phi(g) + \phi(f)\psi(g)\qquad \hbox{for all\ } f,g\in A.$$
A point derivation is said to be \emph{bounded} if it is bounded (continuous) as a linear functional.
It is immediate that the presence of an analytic disc in the maximal ideal space of $A$ implies the existence of a nontrivial Gleason part and nonzero bounded point derivations.

The real part of a complex number (or function) $z$ will be denoted by $\Re z$.   
The Euclidean norm of a point $z\in \CN$ will be denoted by $\|z\|$.




\section{Generalizations of polynomial and rational hulls} \label{definitions}

As mentioned in the introduction, in this section we define extensions of the notions of polynomial and rational hull that are key to the construction of the new examples of hulls without analytic discs presented in this paper.  Throughout the section, $X$ denotes a compact subset of $\CN$.

Recall that Henkin~\cite{Henkin} constructed a Cantor set in $\C^2$ whose rational hull has interior.  A direct generalization of Henkin's result to $\CN$, $N>2$, is not completely satisfying because while when $N=2$, the statement $z\in \hr X$ means that every analytic subvariety of $\C^N$ (of pure positive dimension) passing through $z$ intersects $X$, when $N>2$, the statement $z\in \hr X$ means only that every pure codimension 1 analytic subvariety of $\C^N$ passing through $z$ intersects $X$.  This observation suggests the following extension of the notion of rational hull.  (Here and throughout the paper, by an analytic subvariety of $\CN$ of pure codimension $m$, we mean a pure dimensional analytic subvariety of $\CN$ of pure dimension $N-m$.)

\begin{definition}
For $1\leq k\leq N$, the \emph{$k$-rational hull} $\hrk kX$ of 
$X$ is the set
\begin{eqnarray*}
\hrk kX&\!\!\!=&\!\!\!\{ z\in \CN: {\rm every\ analytic\ subvariety\ of\ } \CN {\rm\ of\ pure\ codimension}\leq k \\
& &\qquad\qquad  {\rm that\ passes\ through\ } z\ {\rm intersects\ } X\}.
\end{eqnarray*}
We say that $X$ is \emph{$k$-rationally convex} if $\hrk kX = X$.
\end{definition}

We define an analogous extension of the notion of polynomial hull as follows.

\begin{definition}
For $2\leq k\leq N$, the \emph{$k$-polynomial hull} $\hk kX$ of 
$X$ is the set
\begin{eqnarray*}
\hk kX&\!\!\!=&\!\!\!\{ z\in \C^N: z\in \hrk {k-1}X\ {\rm and\ } z\in \h {X\cap V}\ {\rm for\ every} \\ 
& &\qquad\qquad{\rm analytic\ subvariety}\ V\ {\rm of\ } \CN {\rm\ of\ pure\  codimension}\leq k-1 \\
& &\qquad\qquad  {\rm that\ passes\ through\ } z\} .
\end{eqnarray*}
The \emph{$1$-polynomial hull} $\hk 1X$ of $X$ is the usual\vadjust{\kern 2pt} polynomial hull $\h X$.
We say that $X$ is \emph{$k$-polynomially convex} if $\hk kX = X$.
\end{definition}

It is immediate from the definitions that 
$$\h X=\hk 1X\supset \hr X=\hrk 1X\supset \hk 2X\supset \hrk 2X 
\supset\cdots\supset \hk NX \supset \hrk NX=X.$$
Also it is easily verified that the $k$-polynomial hull of a compact set is $k$-polynomially convex and the $k$-rational hull is $k$-rationally convex.

One could also consider modifications of the above definitions.

\begin{definition}
For $1\leq k\leq N$, the \emph{quasi-$k$-rational hull} $\khr kX$ of 
$X$ is the set
\begin{eqnarray*}
\khr kX&\!\!\!=&\!\!\!\{ z\in \CN: {\rm if}\ p_1,\ldots, p_k\ {\rm are\ polynomials\ such\ that}  \\
&&p_1(z)=0,\ldots,  p_k(z)=0,\ {\rm then}\ p_1,\ldots, p_k\ {\rm have\ a\ common\ zero\ on}\ X\}.
\end{eqnarray*}
\end{definition}

\begin{definition}
For $2\leq k\leq N$, the \emph{quasi-$k$-polynomial hull} $\kh kX$ of 
$X$ is the set
\begin{eqnarray*}
\kh kX&\!\!\!=&\!\!\!\ \bigl\{ z\in \C^N: z\in \khr {k-1}X\ {\rm and\ } {\rm if}\ p_1,\ldots, p_{k-1}\ {\rm are\ polynomials\ such\ that}  \\
&&p_1(z)=0,\ldots, p_{k-1}(z)=0,\ {\rm then}\ z\in (X \cap\{p_1=0, \ldots,  p_{k-1}=0\})\endhat\,\bigr\}.
\end{eqnarray*}
The \emph{quasi-$1$-polynomial hull} $\kh 1X$ of $X$ is the usual polynomial hull $\h X$.
\end{definition}

The $k$-hulls are ostensibly smaller than the corresponding quasi-$k$-hulls.  Thus the assertion that a set has a nontrivial $k$-hull is stronger than the assertion that the corresponding quasi-$k$-hull is nontrivial.  Therefore, 
we will use exclusively the $k$-hulls.
However, the quasi-$k$-hulls are perhaps more intuitive, and the reader who wishes to do so, may replace the $k$-hulls by the quasi-$k$-hulls throughout.  
Everything will go through essentially unchanged \emph{except} that Theorem~\ref{perfectset} will no longer be obtained and will have to be replaced by an ostensibly weaker statement.

The above hulls seem not to have been considered previously in the literature.  However, the notions of \emph{convexity} corresponding the $k$-rational hulls were introduced by Evgeni Chirka and Stout in~\cite{ChirkaS}.  In addition, the above extensions of the notion of polynomial hull are related to work of Basener~\cite{BasenerShilov} and Nessim Sibony~\cite{Sibony} on generalized boundaries and multidimensional analytic structure:  Applied to the uniform algebra $P(X)$ with $X$ polynomially convex, Basener's generalized Shilov boundary $\partial_k P(X)$ is the smallest closed subset of $X$ whose quasi-$(k+1)$-polynomial hull is equal to $X$.




\section {Some preliminary lemmas}\label{lemmas}

In this section we prove some useful, elementary lemmas regarding $k$-hulls.

\begin{lemma} \label{hullintersection}
Let $\kk$ be a collection of compact sets in $\CN$ totally ordered by inclusion.  Let $K_\infty=\bigcap_{K\in \kk} K$.  Then $\hk k {K}_\infty=\bigcap_{K\in \kk} \hk k K$ and $\hrk k {K_\infty} = \bigcap_{K\in \kk} \hrk k K$.
\end{lemma}

\begin{proof}
The inclusions $\hk k {K}_\infty\subset\bigcap_{K\in \kk} \hk k K$ and $\hrk k {K_\infty}\subset\bigcap_{K\in \kk} \hrk k K$  are obvious.  

To establish the reverse inclusion for $k$-rational hulls,  suppose $x\notin \hrk k {K_\infty}$.  Then there exists a subvariety $V$ of $\CN$ of pure codimension $\leq k$ passing through $x$ and disjoint from $K_\infty$.  Then $\CN\setminus V$ is a neighborhood of $K_\infty$.  Consequently, some set $K\in\kk$ is contained in $\CN\setminus V$, and then $x\notin \hrk k K$.  Thus the result holds for 
$k$-rational hulls.

To complete the proof for $k$-polynomial hulls, suppose now that $x\notin \hk k {K}_\infty$.  If  $x\notin h_r^{k-1} ({K_\infty})$, then by the preceding paragraph, $x\notin \bigcap_{K\in\kk} h_r^{k-1} (K)\supset \bigcap_{K\in\kk} \hk k K$.  Thus we may assume that $x\in h_r^{k-1} ({K_\infty})$.  Then for some subvariety $V$ of $\CN$ of pure codimension $\leq k-1$ passing through $x$, there exists a polynomial $p$ such that $|p(x)|> \|p\|_{K_\infty \cap V}$.  Thus the set $U=\bigl\{w\in V: |p(w)|<|p(x)|\bigr\}$ is a neighborhood of $K_\infty$ in $V$.  Consequently, $U$ contains $K\cap V$ for some $K\in\kk$, and then $x\notin \hk k K$.
\end{proof}

In obtaining our Cantors sets we will make repeated use of the well known characterization of Cantor sets as the compact, totally disconnected, metrizable spaces without isolated points.  The need to verify that our constructions produce sets without isolated points will be avoided by use of the following lemma.  Recall that a subset of a space is said to be {\it perfect\/} if it is closed and has no isolated points.  Every space contains a unique largest perfect subset (which can be empty).  To see this, simply note that the closure of the union of all perfect subsets of a space $X$ is itself a perfect subset of $X$.

\begin{lemma}\label{reductiontop}
Let $X\subset \CN$ be a 
compact set and let $E$ be the largest perfect subset of $X$.  Then $\hhk k E \supset \hhk k X$ and $\hhrk k E\supset\hhrk k X$.
\end{lemma}

\begin{proof}
We first prove the inclusion $\hh E\supset\hh X$. 
Let $\cc$ denote the collection of all closed subsets $C$ of $X$ such that $\hh C \supset \hh X$.  Clearly $\cc$ is nonempty since $X\in \cc$.  If $\cc'$ is a chain in $\cc$, and if $C_\infty=\bigcap_{C\in \cc'} C$, then by the above lemma, $\h C_\infty\supset \hh X$, and hence, $\widehat C_\infty \setminus C_\infty \supset \hh X$.  Thus by Zorn's lemma, 
$\cc$ has a minimal element $F$.  We claim that $F$ has no isolated points.  Assume to the contrary that $F$ has an isolated point, i.e., $F= G\cup \{p\}$ for some closed set $G$ and point $p\notin G$.  It is easily shown that adjoining a single point to a polynomially convex set yields another polynomially convex set \cite[Lemma~2.2]{AIW2001}, so the set $\h G \cup \{p\}$ is polynomially convex.  Thus $\h F = \h G \cup \{p\}$.  Hence $\hh G \supset \hh F$, so $\hh G \supset \hh X$, contradicting the minimality of $F$.  Thus $F$ is a perfect set. 
Consequently, as the {\it largest\/} perfect subset of $X$, the set $E$ contains $F$.   Consequently, $\h E \supset \hh F$, and hence $\hh E \supset \hh X$, as desired.

Now to prove the general case of the inclusion $\hhk k E \supset \hhk k X$, fix $x\in \hhk k X$ and let $V$ be a  subvariety of $\CN$ of pure codimension $\leq k-1$ passing through $x$.  Then $x\in \h {X\cap V}$.  By the result of the preceding paragraph, there is a perfect subset $J$ of $X \cap V$ such that $x\in \h J$.  Then $J\subset E$, and hence, $J\subset E\cap V$.  Therefore, $x\in \h {E\cap V}$.  We conclude that $x\in \hk k E$, and hence, $x\in \hhk k E$, as desired.

The statement for $k$-rational hulls is proved similarly except that in place of \cite[Lemma~2.2]{AIW2001} we must use the following analogue for rational convexity. 
\end{proof}

\begin{lemma}
If $K\subset \CN$ is rationally convex and $z_0\in \CN$, then $K\cup \{z_0\}$ is also rationally convex.
\end{lemma}

\begin{proof}
Assume $z_0\notin K$, since otherwise there is nothing to prove.  Suppose $z\notin K\cup \{z_0\}$.  Then by the rational convexity of $K$, there is a polynomial $p$ such that $p(z)=0$ but $p$ has no zeros on $K$.  If $p(z_0)\neq 0$, we conclude at once that $z\notin h_r(K\cup \{z_0\})$.  Otherwise choose a polynomial $q$ such that $q(z)=0$ but $q(z_0)=1$.  Then for $\vep>0$ small enough, the polynomial $p+\vep q$ has no zeros on $K$, and of course 
$(p +\vep q)(z_0)=\vep\neq 0$ and $(p +\vep q)(z)=0$.
Thus $z\notin h_r(K\cup \{z_0\})$.
\end{proof}




\section{Generalization of the construction\\ of Duval and Levenberg}\label{DuvalL-section}

In this section we use the $k$-hulls defined above to generalize the construction of a polynomial hull without analytic discs due to Duval and Levenberg \cite{DuvalL} and to generalize related constructions given by the present author in \cite{Izzo}.  Throughout the section we make the tacit hypothesis that $N\geq 2$.

As explained in \cite{Izzo}, the construction of Duval and Levenberg is closely related to the construction of a {\em rational} hull with no analytic discs due to Wermer~\cite{Wermerpark} (or see~\cite[Chapter~24]{AW}).
In both the construction of Duval and Levenberg and the construction of Wermer, the set having hull without analytic discs is obtained from the boundary of a domain by removing a sequence of subsets, and the presence of nontrivial hull arises because a point of the domain lies in the rational hull of the set that remains if it does not lie in the polynomial hull of the set removed.  Furthermore, the reason for this is that each point of the domain lies in the $2$-polynomial hull of the boundary of the domain.  This observation is the motivation for the proofs of the main results in this section.

We first state the results, and then turn to their proofs.
We begin with the generalization of the result of Duval and Levenberg.

\begin{theorem}\label{Duval}
Let $\Sigma\subset \CN$ be a compact set, and let $k\geq 2$ be an integer.  If $K\subset {\h \Sigma}^k$ is a compact polynomially convex set, then there is a compact subset $X$ of $\Sigma$ such that $\hrk {k-1} X \supset K$ and the set $\h X\setminus K$ contains no analytic discs.  
\end{theorem}

The next result is a modification yielding a stronger conclusion from a stronger hypothesis.  This generalizes \cite[Theorem~4.1]{Izzo}.

\begin{theorem}\label{Duval-dense}
Let $\Sigma\subset \CN$ be a compact set, and let $k\geq 2$ be an integer.  If $K\subset {\h \Sigma}^k$ is a compact polynomially convex set, and if the set of polynomials that are zero-free on $K$ is dense in $P(\Sigma)$, then there is a compact subset $X$ of $\Sigma$ such that $\hrk {k-1} X \supset K$ and $P(X)$ has dense invertibles.
\end{theorem}

Taking for $K$ a one-point set yields immediately the following.

\begin{corollary}\label{gen-cor}
Let $\Sigma\subset \CN$ be any compact set with nontrivial $k$-polynomial hull with $k\geq 2$.  Then $\Sigma$ contains a compact subset $X$ with nontrivial $(k-1)$-rational hull and such that $P(X)$ has dense invertibles. 
\end{corollary}

The following special case is the case of greatest interest.

\begin{corollary}\label{main-cor}
Let $\Sigma\subset \CN$ be any compact set with nontrivial $2$-polynomial hull.  Then $\Sigma$ contains a compact subset $X$ such that $\h X$ is nontrivial and $P(X)$ has dense invertibles. 
\end{corollary}

We also have analogues of Theorems~\ref{Duval} and~\ref{Duval-dense} with the polynomially convex set $K$ replaced by a rationally convex set.  These generalize \cite[Theorems~1.8 and~4.5]{Izzo}

\begin{theorem}\label{rat-Duval}
Let $\Sigma\subset \CN$ be a compact set, and let $k\geq 2$ be an integer.  If $K\subset {\h \Sigma}^k$ is a compact rationally convex set, then there is a compact subset $X$ of $\Sigma$ such that $\hrk {k-1} X \supset K$ and the set $\hr X\setminus K$ contains no analytic discs.
\end{theorem}

\begin{theorem}\label{rat-Duval-dense}
Let $\Sigma\subset \CN$ be a compact set, and let $k\geq 2$ be an integer.  If $K\subset {\h \Sigma}^k$ is a compact rationally convex set, and if the set of polynomials that are zero-free on $K$ is dense in $P(\Sigma)$, then there is a compact subset $X$ of $\Sigma$ such that $\hrk {k-1} X \supset K$, such that  $R(X)$ has dense invertibles, and such that there is a dense subset of $P(X)$ consisting of polynomials each of which is zero-free on $X$.
\end{theorem}

\begin{remark}
Theorems~\ref{Duval}, \ref{Duval-dense}, \ref{rat-Duval}, and~\ref{rat-Duval-dense} can be generalized by requiring $K$ only to be $l$-polynomially convex (or $l$-rationally convex, as appropriate) for some $1\leq l\leq k$ and weakening the conclusion to the assertion that $\hk lX\setminus K$ (or $\hrk lX\setminus K$) contains no analytic discs.  The interested reader is invited to verify this while reading the proofs of the theorems.
\end{remark}

We begin the proofs of the above theorems with several lemmas.  

\begin{lemma} \label{pkconvex}
Let $X$ be a $k$-polynomially convex set in $\C^N$, let $Y$ be a polynomially convex set in the plane, and let $f$ be a polynomial on $\C^N$.  Then $f^{-1}(Y)\cap X$ is $k$-polynomially convex in $\C^N$.
\end{lemma}

\begin{proof}
Let $x_0\in \C^N\setminus (f^{-1}(Y)\cap X)$ be arbitrary.  We are to show that $x_0\notin (f^{-1}(Y)\cap X)\endhatk$.  Since the desired conclusion is obvious when $x_0\notin \hk kX$, we may assume that $x_0\in X$.  Then $x_0\notin f^{-1}(Y)$.  Since $Y$ is polynomially convex, there is then a polynomial $p$ on $\C$ such that $|p(f(x_0))|> \|p\|_Y$.  But then $|(p\circ f)(x_0)| > \|p\circ f\|_{f^{-1}(Y)}$.  Since $p\circ f$ is a polynomial on $\C^N$, this shows that $x_0\notin (f^{-1}(Y)\cap X)\endhat \supset (f^{-1}(Y)\cap X)\endhatk$.
\end{proof}

The next result is the foundation for the proofs of Theorems~\ref{Duval} and~\ref{Duval-dense}.

\begin{lemma} \label{foundation}
Let $\Sigma\subset\C^N$ be a compact set, let $p$ be a polynomial on $\C^N$, and let $X=\{\Re p\leq 0\} \cap \Sigma$.  Let $k\geq 2$ be an integer.  Then 
$\{\Re p\leq 0\} \cap \hk {k-1}\Sigma \supset \hk {k-1}X \supset  \hrk {k-1}X \supset
\{\Re p\leq 0\} \cap \hk k\Sigma$.
\end{lemma}


\begin{proof}
First note that for every real number $\alpha$, the sets $\{\Re p\leq \alpha\} \cap \hk {k-1}\Sigma$ and $\{\Re p\geq \alpha\} \cap \hk {k-1}\Sigma$ are 
$(k-1)$-polynomially convex by the preceding lemma.  Therefore,
$$\{\Re p\leq 0\} \cap \hk {k-1}\Sigma \supset \hk {k-1}X \supset  \hrk {k-1}X .$$
To show that $\hrk {k-1}X \supset \{\Re p\leq 0\} \cap \hk k\Sigma$, suppose $z_0$ is a point of $\hk k\Sigma$ such that $z_0\notin \hrk {k-1}X$.   Then there is an analytic subvariety $V$ of $\CN$ of pure codimension $\leq k-1$ passing through $z_0$ and disjoint from $X$.  Then $\Re p>0$ everywhere on the set $V\cap \Sigma$.  Hence there is an $\alpha>0$ such that $V\cap  \Sigma\subset \{\Re p\geq \alpha\} \cap \Sigma$.  Because $z_0$ is in $\hk k\Sigma$, it follows immediately that $z_0$ is in the polynomial hull of $\{\Re p\geq\alpha\} \cap\Sigma$ and hence lies in $\{\Re p\geq\alpha\} \cap \h \Sigma$.  Thus $z_0\notin 
\{\Re p\leq 0\} \cap \hk k\Sigma$.
\end{proof}

Note that if $\hk {k-1}\Sigma = \hk k\Sigma$, then all the inclusions in the conclusion of the lemma are equalities.

For a short proof of the next lemma see \cite[Lemma~1.7.4]{Stout}.  

\begin{lemma}\label{Stout-lemma}
If $X\subset \CN$ is a polynomially convex set, and if $E\subset X$ is polynomially convex, then for every holomorphic function $f$ defined on a neighborhood of $X$, the set $E\cup (X\cap f^{-1}(0))$ is polynomially convex.
\end{lemma}

The proof of the next lemma is based on the original argument of Duval and Levenberg \cite{DuvalL} as presented in \cite[Lemma~1.7.5]{Stout} and contains the main technical construction for the proofs of Theorems~\ref{Duval} and~\ref{Duval-dense}.

\begin{lemma} \label{keylemma}
Let $\Sigma\subset \CN$ be a compact set, and let $k\geq 2$ be an integer. If $K\subset \hk k\Sigma$ is a compact polynomially convex set, and if $\{p_j\}_{j=1}^\infty$ is a sequence of polynomials such that each of the sets $Z_j=\h\Sigma\cap \pji(0)$ is disjoint from $K$, then there is a compact subset $X$ of $\Sigma$ such that $\hrk {k-1}X\supset K$ and  $\h X\cap Z_j=\emptyset$ for every $j$.
\end{lemma}

\begin{proof}
We will construct a sequence of polynomials $\{f_j\}_{j=1}^\infty$ such that for each $j$ we have $\Re f_j<0$ on $K$ and $\Re f_j>0$ on $Z_j$ and such that the sets 
\be
X_j=\{\Re f_j\leq 0\}\cap \Sigma\label{Xj}
\ee
form a decreasing sequence.  The set
$X=\bigcap_{j=1}^\infty X_j$
then has the required properties by Lemmas~\ref{hullintersection} and~\ref{foundation}.

We construct the $f_j$ inductively.  First note that the set $K\cup Z_1$ is polynomially convex by Lemma~\ref{Stout-lemma}, so there is a polynomial $f_1$ such that
\[ \hbox{$\Re f_1<-1$ on $K$ and $\Re f_1>1$ on $Z_1$}.\]
Define $X_1=\{\Re f_1\leq 0\}\cap \Sigma$.  

For the inductive step, assume that polynomials $f_1,\ldots,f_n$ have been chosen such that the sets $X_1,\ldots, X_n$ defined as in (\ref{Xj}) for $j=1,\ldots, n$ form a decreasing sequence, and for each $j$ we have $K\subset \{\Re f_j<-1\}$ and $Z_j\subset \{\Re f_j>1\}$.  Let $L_n=\{\Re f_n\geq 0\}\cap \h \Sigma$.  This set is polynomially convex.  Because the sets $f_n(K)$ and $f_n(L_n)$ lie in disjoint half-planes, their polynomial hulls are disjoint.  Therefore, $K\cup L_n$ is polynomially convex by Kallin's lemma \cite{Kallin} (or see \cite[Theorem~1.6.19]{Stout}).  Now another application of Lemma~\ref{Stout-lemma} yields that the set $K\cup L_n \cup Z_{n+1}$ is polynomially convex.  Consequently, there is a polynomial $f_{n+1}$ such that $\Re f_{n+1}<-1$ on $K$ and $\Re f_{n+1}> 1$ on $L_n\cup Z_{n+1}$.  Now define $X_{n+1}$ as in (\ref{Xj}) with $j=n+1$ and observe that the sets $X_1,\ldots, X_{n+1}$ then have the required properties for the induction to continue.
\end{proof}

With the lemmas in hand, we now conclude the proofs of Theorems~\ref{Duval} and~\ref{Duval-dense}.

\bpf[Proof of Theorem~\ref{Duval}]
The collection of all polynomials on $\C^N$ having no zeros on $K$, viewed as a subset of $P(\Sigma)$, has a countable dense subset.  Choosing such a subset $\{p_j\}$ and applying the lemma just proven yields a compact subset $X$ of $\Sigma$ such that $\hrk {k-1}X\supset K$ and each $p_j$ is zero-free on $\h X$.  
Assume to get a contradiction that $\h X\setminus K$ contains an analytic disc $\sigma:D\rightarrow \h X\setminus K$.  Since $K$ is polynomially convex, and hence rationally convex, there is a polynomial $p$ such that $p(\sigma(0))=0$ and $p$ has no zeros on $K$.  Because $\sigma$ is injective, we may assume, by adding to $p$ a small multiple of a suitable first degree polynomial if necessary,
that  $\bigl(\partial (p\circ \sigma)/\partial z\bigr)(0)\not= 0$.  Then Rouche's theorem shows that every polynomial that is uniformly close to $p$ on $\sigma(D)$ also has a zero on $\sigma(D)$.  But then some $p_j$ must have a zero on $\sigma(D)\subset \h X$, a contradiction. 
\epf

\bpf[Proof of Theorem~\ref{Duval-dense}]
The hypotheses enable us to choose a countable dense subset $\{p_j\}$ of $P(\Sigma)$ consisting of polynomials each of which is zero-free on $K$.  The result then follows immediately from Lemma~\ref{keylemma}.
\epf

For the proofs of Theorems~\ref{rat-Duval} and~\ref{rat-Duval-dense} we need analogues of some of the lemmas above.

\blem\label{rat-foundation}
Let $\Sigma\subset\C^N$ be a compact set, and let $k\geq 2$ be an integer.
Suppose $\{p_j\}$ is a sequence of polynomials and $\{r_j\}$ is a sequence of strictly positive numbers.
Define 
$$X=\Sigma\setminus \bigcup_{j=1}^\infty \{z\in \Sigma: |p_j(z)|< r_j\}$$
 and for each $j$,
$$H_j=\{z\in \Sigma: |p_j(z)|< r_j\}.$$
For $z_0\in \hk k \Sigma$, if $z_0\notin \hrk {k-1}X$, then for some $n$, we have $z_0\in [\cup_{j=1}^n \oh_j]\endhat$.
\elem

\bpf
Suppose $z_0\notin \hrk {k-1}X$.  Then there is a subvariety $V$ of $\CN$ of pure codimension $\leq k-1$ passing through $z_0$ and disjoint from $X$.  Then $V\cap \Sigma\subset \cup_{j=1}^\infty H_j$.  Since $V\cap \Sigma$ is compact, there is an $n$ such that $V\cap \Sigma \subset \cup_{j=1}^n H_j$.  Because $z_0$ is in $\hk k\Sigma$, it follows immediately that $z_0$ is in the polynomial hull of $\cup_{j=1}^n \oh_j$
\epf

\blem\label{rat-keylemma}
Let $\Sigma\subset\C^N$ be a compact set, and let $k\geq 2$ be an integer.
If $K\subset \hk k\Sigma$ is a rationally convex set, and if $\{p_j\}$ is a sequence of polynomials such that each of the sets $Z_j=\h \Sigma\cap \pji(0)$ is disjoint from $K$, then there is a compact subset $X$ of $\Sigma$ such that $\hrk {k-1}X \supset K$ and $\hr X \cap Z_j=\emptyset$ for all $j$.
\elem

\bpf
The proof is the same as the proof of \cite[Lemma~3.5]{Izzo} with \cite[Lemma~3.4]{Izzo} replaced by Lemma~\ref{rat-foundation}, so we omit the details.  
One shows that for a suitable choice of strictly positive numbers $r_1, r_2, \ldots$, the set $X$ defined as in Lemma~\ref{rat-foundation} satisfies $\hrk {k-1}X\supset K$.  Since it is obvious that each $p_j$ is zero-free on $X$, that proves the lemma.
\epf

\bpf[Proof of Theorem~\ref{rat-Duval}]
This follows from Lemma~\ref{rat-keylemma} by exactly the same argument that was used above to obtain Theorem~\ref{Duval} from 
Lemma~\ref{rat-keylemma}.
\epf

\bpf[Proof of Theorem~\ref{rat-Duval-dense}]
The proof is similar to the proof of Theorem~\ref{Duval-dense}.  The hypotheses enable us to choose a countable dense subset  $\{ p_j\}$ of $P(\Sigma)$ consisting of polynomials  each of which is zero-free on $K$.  The result then follows immediately from Lemma~\ref{rat-keylemma} and the following lemma whose easy proof is given in \cite[Lemma~4.6]{Izzo}.
\epf

\blem \label{poly-to-rat}
Let $X\subset \C^N$ be compact.  Suppose there is a dense subset of $P(X)$ consisting of polynomials each of which is zero-free on $X$.  Then $R(X)$ has dense invertibles.
\elem




\section{A Cantor set in $\CN$ with large 
$(N-1)$-rational hull}\label{Cantor-set-section}

In this section we prove the existence of a Cantor set in $\CN$ whose ${(N-1)}$-rational hull has interior.  Note that, in the context of the hulls considered in this paper, this is the strongest possible statement about Cantor sets having large hulls, for since every Cantor set in the complex plane is polynomially convex, every Cantor set in $\CN$ is $N$-polynomially convex.

\begin{theorem}\label{CantorLhull} 
For $N\geq 2$, there exists a Cantor set $X$ in $\CN$ such that $\hrk {N-1} X$ contains the closed unit ball of $\CN$.  Furthermore, for any $\rho>1$, the set $X$ can be taken to lie in the spherical shell $\spshell$.
\end{theorem}

Before proving this theorem, we establish the following easy 
corollary.

\begin{theorem}\label{perfectset} 
For $N\geq 2$, there exists a totally disconnected, perfect set in $\CN$ that intersects every analytic subvariety of $\CN$ of positive dimension.
\end{theorem}

\begin{proof}
Theorem~\ref{CantorLhull} gives that for each positive integer $n$ there exists a Cantor set $X_n$ contained in the spherical shell $\{z\in \CN: n<|z|<n+\frac{1}{2}\}$ such that $\hrk {N-1} {X_n}$ contains the closed ball of radius $n$ centered at the origin.  Let $X=\bigcup_{n=1}^\infty X_n$.  Then $X$ is easily seen to be a totally disconnected, perfect set in $\CN$ that intersects every analytic subvariety of $\CN$ of positive dimension.
 \end{proof}

The proof of Theorem~\ref{CantorLhull} was inspired by Vitushkin's construction of a Cantor set in $\C^2$ whose polynomial hull has interior \cite{Vit} and Henkin's construction of a Cantor set in $\C^2$ whose \emph{rational} hull has interior \cite{Henkin}.  In outline, the presentation 
given here follows the exposition of Vitushkin's construction given in 
\cite[Chapter~21]{Wermerbook}, while the details are more closely related to Henkin's construction.

The boundary of the open unit ball $B=\{z: \|z\|<1\}$ in $\CN$ will be denoted by $\partial B$.

The following lemma is a consequence of Lemma~\ref{foundation} because $\h {\partial B} = \hk N{\partial B}=\oB$.  (A direct proof similar to the proof of \cite[Lemma~3.2]{Izzo} can also be given.)

\begin{lemma}\label{hullofcap}
Let $N\geq 2$.  Let $p$ be a polynomial on $\CN$, and let $X=\{\Re p\leq 0\} \cap \pB$.  Then $\h X = \hrk {N-1} X = \{\Re p\leq 0\} \cap \oB$.
\end{lemma}

We will denote by $S_R$ the sphere in $\CN$ of radius $R$ centered at the origin: $S_R=\{z\in \CN: \|z\|=R\}$.  By a {\it spherical cap\/} on the sphere $S_R$ we mean a set of the form
$$E=\{z\in S_R: \Re\< z,\zeta\>\geq t\}$$
for some $\zeta\in \pB= S_1$ and $t<R$.  (Here $\<z,\zeta\>$ denotes the usual Hermitian inner product on $\CN$.)  We require $t<R$ so that every spherical cap on $S_R$ has nonempty interior in $S_R$.  Note that when $t\leq -R$, the spherical cap $E$ above is $S_R$ itself.

\begin{lemma}\label{littlecaps}
Let $N\geq 2$.
Let $0<R<\rho$ be given, let $E$ be a spherical cap on the sphere $S_R$ in $\CN$, let $\Omega$ be a neighborhood of $E$ in $\CN$, and fix $\ep>0$.  Then there exists a finite set of pairwise disjoint spherical caps $E^1,\ldots, E^s$ with each $E^j$ on a sphere $S_{R_j}$ with $R<R_j<\rho$ such that
{
\item{\rm (i)} $E^j\subset \Omega,\ j=1,\ldots, s$ \openup1.8\jot
\item{\rm (ii)} $\diam(E^j)<\ep,\ j=1,\ldots s$
\item{\rm (iii)} $h^{N-1}_r\biggl(\textstyle\bigcup\limits_{k=1}^s E^j\biggr) \supset \hrk {N-1} E$.

}
\end{lemma}

\begin{proof}
By compactness of $E$, there is a finite collection of spherical caps
$G^1,\ldots, G^s$ on $S_R$ that cover $E$ each having diameter less than $\ep$ and contained in $\Omega$.  Let $\rowonly \zeta s \in \pB$ and $\rowonly ts \in \R$ be such that
$$G^j=\{z\in S_R: \Re\<z,\zeta_j\>\geq t_j\}.$$
Then there exists an $R'$ with $R<R'<\rho$ such that for each $j=1,\ldots, s$ the set
$$H^j=\{z:R\leq \|z\|\leq R' \ {\rm and}\ \Re\<z,\zeta_j\>\geq t_j\}$$
has diameter less than $\ep$ and is contained in $\Omega$.  Choose distinct numbers $\rowonly Rs$ such that $R<R_j<R'$ for each $j=1,\ldots, s$, and set
$$E^j=\{z\in S_{R_j}: \Re\<z,\zeta_j\>\geq t_j\}.$$
Then $E^j\subset H^j$, so conditions~(i) and~(ii) hold. Condition (iii) holds also since an application of Lemma~\ref{hullofcap} shows that $\hrk {N-1} {E^j} \supset G^j$, and hence,
$$h^{N-1}_r\biggl(\textstyle\bigcup\limits_{k=1}^s E^j\biggr) \supset h^{N-1}_r\biggl(\textstyle\bigcup\limits_{k=1}^s G^j\biggr) \supset \hrk {N-1} E.$$
\end{proof}

\begin{lemma}\label{Wermeranalogue}
Let $N\geq 2$.
Given $\rho>1$, there exists a sequence of compact subsets $(K_n)$ of $\CN$ each contained in the spherical shell $\spshell$ such that
{
\item{\rm(i)} $K_{n}\subset K_{n-1}$ for each $n>1$\openup1\jot
\item{\rm(ii)} $\hrk {N-1} {K_n} \supset \oB$ for each $n\geq 1$
\item{\rm(iii)} For each $n>1$, there exist disjoint closed sets $K_n^1,\ldots, K_n^{s_n}$ such that $\diam K_n^j<1/n$ for each $j$ and $\bigcup_{j=1}^{s_n} K_n^{s_n}= K_n$.

}
\end{lemma}

\begin{proof}
Throughout the proof, all spherical caps are to be chosen to lie on spheres centered at the origin.  

Choose $R$ such that $1<R<\rho$.  Let $E_1$ be the sphere $S_{R}$, and let $K_1$ be a compact neighborhood of $E_1$.  Then $\hrk {N-1} {K_1}\supset \oB$.  

By Lemma~\ref{littlecaps}, we can choose disjoint spherical caps $E_2^1,\ldots, E_2^{s_2}$ such that each $E_2^j$ is contained in the interior of $K_1$, $\diam (E_2^j)<1/2$, and
\be
h^{N-1}_r \biggl( {\textstyle\bigcup\limits_{j=1}^{s_2} E_2^j} \biggr) \supset \hrk {N-1} {E_1} \supset \oB. \label{hull12}  
\ee
Next choose disjoint compact neighborhoods $K_2^j$ of $E_2^j$, for each $j$, so that $K_2^j\subset K_1$ and $\diam K_2^j<1/2$.  Set $K_2=\bigcup_{j=1}^{s_2} K_2^j$.  Then $K_2\subset K_1$.  Also $\hrk {N-1} {K_2}\supset \hrk {N-1} {\bigcup_{j=1}^{s_2} E_s^j} \supset \oB$.  Thus (i), (ii), and (iii) hold for $n=2$.

Now fix $j$ with $1\leq j \leq s_2$.  By Lemma~\ref{littlecaps} we can choose disjoint spherical caps $E_3^{j1},\ldots, E_3^{j t_j}$ such that $E_3^{jk}$ is contained in the interior of $K_2^j$ and $\diam E_3^{jk}< 1/3$ for each $k$, and
\be
h^{N-1}_r \biggl( {\textstyle\bigcup\limits_{k=1}^{t_j} E_3^{jk}} \biggr) \supset E_2^j.
\label{hull23}
\ee

Next choose disjoint compact neighborhoods $K_3^{jk}$ of $E_3^{jk}$ such that $K_3^{jk}\subset K_2^j$ and $\diam K_3^{jk} < 1/3$ for all $k$.  Set
$$K_3=\textstyle\bigcup\limits_{j=1}^{s_2} \textstyle\bigcup\limits_{k=1}^{t_j} K_3^{jk}.$$
Then $K_3\subset K_2$.  Also $\hrk {N-1} {K_3} \supset E_2^j$ for each $j$ by (\ref{hull23}), and hence
$$
\hrk {N-1} {K_3} \supset h^{N-1}_r \biggl( {\textstyle\bigcup\limits_{j=1}^{s_2} E_2^j} \biggr) \supset \oB
$$
by (\ref{hull12}).  Thus (i), (ii), and (iii) hold for $n=3$.

By continuing, one obtains the desired sequence $(K_n)$.
\end{proof}

\begin{proof}[Proof of Theorem~\ref{CantorLhull}]
Let $(K_n)$ be a sequence of sets as in the preceding lemma, and set $K_\infty= \bigcap_{n=1}^\infty K_n$.  Then $K_\infty$ is a compact subset of the spherical shell $\spshell$, and $\hrk {N-1} {K_\infty}\supset \oB$ by Lemma~\ref{hullintersection}.  Now let $Y$ be a connected subset of $K_\infty$.  Then for each $n>1$, we have $Y\subset K_n^j$ for some $j$, and hence $\diam Y< 1/n$.  Consequently, $Y$ is a single point.  Thus $K_\infty$ is totally disconnected.

Now let $X$ be the largest perfect subset of $K_\infty$.  Then by Lemma~\ref{reductiontop}, $\hrk {N-1} X \setminus X \supset \hrk {N-1} {K_\infty} \setminus K_\infty \supset \oB$, and finally, $X$ is a Cantor set by the well known characterization of Cantor sets as the compact, totally disconnected, metrizable spaces without isolated points.
\end{proof}




\section{Cantor sets with hulls that contain no analytic discs}\label{Cantor-no-discs}

In this section we prove the existence of Cantor sets with nontrivial hulls without analytic discs and show that we can impose various additional conditions on these Cantor sets.

We begin with the following result which contains Theorems~\ref{Cantorset} and ~\ref{Cantorset-dense-invert} stated in the introduction.

\bthm\label{Cantorset-k}
For $N\geq 3$, there exists a Cantor set $X$ in $\CN$ such that $\hrk {N-2}X\setminus X$ is nonempty and $P(X)$ has dense invertibles.
\ethm

\bpf
By Theorem~\ref{CantorLhull}, there exists a Cantor set $\Sigma$ in $\CN$ such that $\hrk {N-1}\Sigma$ contains the closed unit ball of $\CN$.  Let $K$ be a one-point subset of $\hrk {N-1}\Sigma\setminus \Sigma$, and apply Theorem~\ref{Duval-dense} to get the existence of a compact subset $Y$ of $\Sigma$ such that $\hrk {N-2}Y \supset K$ and $P(Y)$ has dense invertibles.
Now let $X$ be the largest perfect subset of $Y$.  Then $\hrk {N-2}X\setminus X \supset \hrk {N-2}Y \setminus Y \neq \emptyset$ by Lemma~\ref{reductiontop}.  The set $X$ is a Cantor set since it is a perfect subset of the Cantor set $\Sigma$.  Finally, an easy argument shows that the condition that $P(Y)$ has dense invertibles  implies that $P(X)$ has dense invertibles as well.
\epf

As discussed in the introduction, we can also handle the question of the existence of nontrivial Gleason parts and nonzero bounded point derivations in hulls of Cantor sets.  The next result shows these hulls can have no nontrivial Gleason parts and no nonzero bounded point derivations, while the results presented after it show that these hulls can also have large Gleason parts and an abundance of nonzero bounded point derivations.

\begin{theorem}\label{no-derivations}
There exists a Cantor set $X$ in $\C^4$ such that
\item{\rm(i)} $\h X\setminus X$ is nonempty
\item{\rm(ii)} $P(X)$ has dense invertibles
\item{\rm(iii)} every point of $\h X$ is a one-point Gleason part for $P(X)$
\item{\rm(iv)} there are no nonzero bounded point derivations for $P(X)$.
\end{theorem}

\begin{theorem}\label{ptderivation1}
For each integer $N\geq 3$, there exists a Cantor set $X$ in $\CN$ such that
\item{\rm(i)} $\hrk {N-2}X\setminus X$ has positive $2N$-dimensional measure
\item{\rm(ii)} $P(X)$ has dense invertibles
\item{\rm(iii)} there is a set $P\subset \hrk {N-2}X\setminus X$ of positive 
$2$-dimensional Hausdorff measure contained in a single Gleason part for $P(X)$ 
\item{\rm(iv)} at each point of $P$ there is a nonzero bounded point derivation for $P(X)$.
\end{theorem}

In the case $N\geq 4$ we can obtain a larger Gleason part and more bounded point derivations, but at the expense of giving up assurance that $\h X\setminus X$ has positive $2N$-dimensional measure.

\begin{theorem}\label{ptderivation2}
For each integer $N\geq 4$, there exists a Cantor set $X$ in $\CN$ such that
\item{\rm(i)} $P(X)$ has dense invertibles
\item{\rm(ii)} there is a set $P\subset \h X\setminus X$ of positive 
$2(N-1)$-dimensional Hausdorff measure contained in a single Gleason part for $P(X)$ 
\item{\rm(iii)} at each point of $P$ the space of bounded point derivations for $P(X)$ has dimension at least~$N-1$.
\end{theorem}

For rational hulls we have the following.

\begin{theorem}\label{ptderivation3}
For each integer $N\geq 3$, there exists a Cantor set $X$ in $\CN$ such that 
\item{\rm(i)} $R(X)$ has dense invertibles
\item{\rm(ii)} there is a set $P\subset \hrk {N-2}X\setminus X$ of positive $2N$-dimensional measure contained in a single Gleason part for $R(X)$ 
\item{\rm(iii)} at each point of $P$ the space of bounded point derivations for $R(X)$ has dimension~$N$.
\end{theorem}

Using the Cantor set of Theorem~\ref{CantorLhull} as in the proof of Theorem~\ref{Cantorset-k}, one can prove
Theorems~\ref{ptderivation1} and~\ref{ptderivation3} in essentially the same manner as \cite[Theorems~1.1 and~1.3]{Izzo}, with \cite[Theorems~4.1 and~4.5]{Izzo} replaced by Theorems~\ref{Duval-dense} and~\ref{rat-Duval-dense}.  The details are left to the reader.
Theorem~\ref{ptderivation2} follows from
Theorem~\ref{ptderivation3} by a standard method for translating 
statements about rational hulls in $\C^N$ into statements about polynomial hulls in $\C^{N+1}$.  See the end of Section~2 of \cite{Izzo}.

To establish Theorem~\ref{no-derivations} on the existence of a Cantor set $X$ with polynomial hull 
such that $P(X)$ has no nontrivial Gleason parts and no nonzero bounded point derivations, we will prove the following general result using the methods in \cite{CGI}.

\bthm\label{general-no-derivations}
Let $Y$ be a compact set in $\CN$ such that $\h Y \setminus Y$ is nonempty and $P(Y)$ has dense invertibles.  Then there exists a compact set $X$ in $\C^{N+1}$ such that,
letting $\pi$ denote the restriction to $\h X$ of the projection $\C^{N+1} \to \CN$ onto the first $N$ coordinates, the following conditions hold:
\item {\rm (i)} $\pi (X) = Y$
\item {\rm (ii)} $\pi (\h X \setminus X) = \h Y \setminus Y$
\item{\rm (iii)} $P(X)$ has dense invertibles
\item {\rm (iv)} every point of $\h X$ is a one-point Gleason part for $P(X)$ and there are no nonzero bounded point derivations for $P(X)$
\item {\rm (v)} each fiber $\pi^{-1} (z)$ for $z \in \h Y$ is a Cantor set.
\ethm

Before proving Theorem~\ref{general-no-derivations} we note that Theorem~\ref{no-derivations} follows as a special case.

\bpf[Proof of Theorem~\ref{no-derivations}]
By Theorem~\ref{Cantorset-k}, there is a Cantor set $Y$ in $\C^3$ such that $\h Y \setminus Y$ is nonempty and $P(Y)$ has dense invertibles.  Let $X$ be the set in $\C^4$ obtained from $Y$ as in Theorem~\ref{general-no-derivations}.  Conditions~(i) and~(v) then give that $X$ is totally disconnected and has no isolated points, so $X$ is a Cantor set.  Conditions~(ii),~(iii) and~(iv) give the other properties required of $X$.
\epf

\bpf[Proof of Theorem~\ref{general-no-derivations}]
Much of the argument is a repetition of the proof of \cite[Theorem~1.1]{CGI}, so we omit some of the details.
Set 
$X_0=\h Y$ and $A_0=P(X_0)$.  Using an iterative procedure introduced by Cole \cite[Theorem~2.5]{Cole}, we will define a sequence of uniform algebras 
$\{A_m\}_{m=0}^\infty$.  First let $\sF_0=\{f_{0,n}\}_{n=1}^\infty$ be a countable dense set of invertible functions in $A_0$.  Let $p:X_0\times \C^\omega\rightarrow X_0$ and $p_n:X_0\times \C^\omega \rightarrow \C$ denote the projections given by $p\bigl(x,(y_k)_{k=1}^\infty\bigr)=x$ and $p_n\bigl(x,(y_k)_{k=1}^\infty\bigr)=y_n$.  Define $X_1\subset X_0\times \C^\omega$ by
$$X_1=\{z\in X_0\times \C^\omega: p_n^2(z)= (f_{0,n}\circ p)(z) \hbox{\ for all } n=1, 2, \ldots\},$$
and let $A_1$ be the uniform algebra on $X_1$ generated by 
the functions $\{p_n\}_{n=1}^\infty$.
Since $p_n^2=f_{0,n}\circ p$ on $X_1$, the functions $f_{0,n}\circ p$ belong to $A_1$.  Identifying each function $f\in A_0$ with $f\circ p$, we can regard $A_0$ as a subalgebra of $A_1$.  Then $p_n^2=f_{0,n}$, so we will denote $p_n$ by $\sqrt{f_{0,n}}$.  
As in \cite{CGI}, there is a countable dense subset 
$\sF_1=\{f_{1,1}, f_{1,2}, \ldots\}$ of $A_1$ that contains the set 
$\{\sqrt{f_{0,n}}\}_{n=1}^\infty$ and consists exclusively of invertible elements of $A_1$ each a polynomial in the members of $\{\sqrt{f_{0,n}}\}_{n=1}^\infty$.
Iterating this construction we obtain a sequence of uniform algebras $\{A_m\}_{m=0}^\infty$ on compact metrizable spaces $\{X_m\}_{m=0}^\infty$ and for each $m$ a countable dense set $\sF_m=\{f_{m,1}, f_{m,2}, \ldots\}$ of invertible functions in $A_m$.  Each $A_m$ can be regarded as a subalgebra of $A_{m+1}$.  
Each function in $\sF_{m+1}$ is a polynomial in the members of 
$\{\sqrt{f_{m,n}}\}_{n=1}^\infty$.  In addition, each $\sqrt{f_{m,n}}$ lies in $\sF_{m+1}$.

We now take the direct limit of the system of uniform algebras $\{A_m\}$ to obtain a uniform algebra $A_\omega$ on some compact metrizable space $X_\omega$.  If we regard each $A_n$ as a subset of $A_\omega$ in the natural way, and set $\sF=\bigcup \sF_m$, then $\sF$ is a dense set of invertibles in $A_\omega$, and every member of $\sF$ has a square root in $\sF$.  It follows (by \cite[Lemma~1.1]{Cole}) that every Gleason part for $A_\omega$ consists of a single point and that there are no nonzero bounded point derivations on $A_\omega$.

By \cite[Theorem~2.5]{Cole}, the maximal ideal space of $A_\omega$ is $X_\omega$, and there is a surjective map $\tpi:X_\omega\rightarrow X_0$ that sends the Shilov boundary for $A_\omega$ into the Shilov boundary for $A_0=P(\h Y)$.  
As in \cite{CGI} there is
a function $b$ such that the functions $z_1\circ \tpi,\ldots, z_N\circ \tpi, b$ generate $A_\omega$ as a uniform algebra.   Setting $f_1= z_1\circ \tpi,\ldots f_N=z+N\circ \tpi, f_{N+1}=b$, and setting 
$X=\{ (f_1(x),\ldots , f_{N+1}(x)): x\in \tpi^{-1}(Y)\}$, we have that $P(X)$ is isomorphic as a uniform algebra to $A_\omega$.  Consequently,
$\h X=\{ (f_1(x),\ldots , f_{N+1}(x)): x\in X_\omega\}$, and~(i) and~(ii) hold.  That conditions (iii) and (iv) hold follows from the corresponding statement for $A_\omega$.

It remains to establish condition (v).  We have the inverse system of compact spaces
\[ \cdots \stackrel{\tau_{m+2}}{\longrightarrow} X_{m+1} \stackrel{\tau_{m+1}}{\longrightarrow} X_{m} \stackrel{\tau_{m}}\longrightarrow \cdots \stackrel
{\tau_2}{\longrightarrow} X_1
\stackrel{\tau_1}{\longrightarrow} X_0,
\]
where 
$$X_{m+1}=\{z\in X_m\times \C^\omega: p_n^2(z)= (f_{m,n}\circ p)(z) \hbox{\ for all } n=1, 2, \ldots\},$$
and $\tau_{m+1}(x, (y_k))_{k=1}^\infty=x$.  
The space $X_\omega$ is the inverse limit of this inverse system:
$$X_\omega=\{(z_k)_{k=0}^\infty\in \prod_{k=0}^\infty X_k: \tau_{m+1}(z_{m+1})=z_m  \hbox{\ for all } m=0, 1, 2, \ldots\}.$$
Because each function $f_{m,n}$ is zero-free, each fiber of each $\tau_m$ is a countable product of two-point spaces and hence is a Cantor set.  
Condition (v) is equivalent to the statement that each fiber of $\tpi$ is a Cantor set.  Furthermore, letting $\pi_m: \prod\limits_{k=0}^\infty X_k \rightarrow X_m$ denote the usual projection, the map $\tpi$ is the restriction of $\pi_0$ to $X_\omega$.

Let $x_0\in X_0$ be arbitrary.  Because $X_\omega$ is compact and metrizable, to show that $\tpi^{-1}(x_0)$ is a Cantor set, it suffices to show that $\tpi^{-1}(x_0)$ is totally disconnected and has no isolated points.  

Consider an arbitrary connected subset $C$ of $\tpi^{-1}(x_0)$.  We show by induction that $\pi_m(C)$ is a singleton for each $m$, and hence $C$ itself is a singleton.  First note that of course $\pi_0(C)=\tpi(C)=\{x_0\}$.  Now assume that $\pi_m(C)$ is a singleton $\{x_m\}$.  Then $\pi_{m+1}(C)$ is contained in the set $\tau_{m+1}^{-1}(x_m)$.  Since each fiber of $\tau_{m+1}$ is a Cantor set and $\pi_{m+1}(C)$ is connected, we conclude that $\pi_{m+1}(C)$ is a singleton, as desired.

To show that $\tpi^{-1}(x_0)$ has no isolated points, consider a point $a=(a_k)_{k=0}^\infty$ in $\tpi^{-1}(x_0)$ and a neighborhood $U$ of $a$ in $\prod X_k$ having the form $U= U_1\times \cdots \times U_m \times X_{m+1} \times X_{m+2} \times \cdots$ with each $U_k$ open in $X_k$.  We are to show that there exists a point $b\neq a$ in $\tpi^{-1}(x_0) \cap U$.  Set $b_k=a_k$ for $k=0,\ldots, m$.  Choose a point $b_{m+1}$ in $\tau_{m+1}^{-1}(b_m)$ distinct from $a_{m+1}$. Because each $\tau_k$ is surjective, there exist points $b_k$, $k=m+2, m+3, \ldots$, such that setting $b=(b_k)_{k=0}^\infty$, the point $b$ is in $X_\omega$.  Of course $b$ is distinct from $a$ and lies in $\tpi^{-1}(x_0)\cap U$, as desired.
\epf




\section {Spaces with hulls that contain no analytic discs}\label{spaces-no-discs}

 In this section we prove the results stated in the introduction about arcs, curves, and general spaces with hulls without analytic discs in the following order:  
 Theorem~\ref {generalspace}, 
 Theorems~\ref{maintheorem} and~\ref{one-d-gen}, Corollaries~\ref{maincorollary} and~\ref{one-d-special}, Theorem~\ref {arc}.
We denote the maximal ideal space of a uniform algebra $A$ by $\mm_A$.

\begin {proof}  [Proof of Theorem~\ref {generalspace}]      
Let $J$ be a Cantor set contained in $K$.   By Theorem~\ref{Cantorset}, there exists a uniform algebra $B$ on $J$ such that $\mm_B\bs J$ 
is nonempty but $\mm_B$ contains no analytic discs.  Let $A$ be the uniform algebra on $K$ defined by
$A= \{f\in C(K): f|J\in B\}$.

Each of $K$ and $\mm_B$ can be regarded as subsets of $\mm_A$ in standard ways.  Then $\mm_B$ is the $A$-convex hull of $J$ in $\mm_A$ \cite[Theorem~II.6.1]{Gamelin}, and it follows that $K \cap \mm_B=J$.

Let $\Sigma=K \cup \mm_B$.  We claim that $\Sigma=\mm_A$.  To see this, let $\widehat B$ denote the uniform algebra on $\mm_B$ obtained from $B$ via the Gelfand transform, and let $\widetilde A=\{f\in C(\Sigma):f|\mm_B\in \widehat B\}$.
Then by \cite[Theorem~4]{Bear}, $\Sigma=\mm_{\widetilde A}$.  Since the map $\widetilde A\to A$ given by restriction $(f \mapsto f|K)$ is an isometric isomorphism, this gives that $\Sigma= \mm_A$.

We conclude that if $\mm_A$ contains an analytic disc, then either $\mm_B$ or $\mm_A \bs \mm_B = K \bs J$ must contain an analytic disc.  But $\mm_B$ contains no analytic discs by our choice of $B$, and $K \bs J$ contains no analytic discs because the real-valued functions in $A$ separate points on $K \bs J$.

Finally note that $\mm_A \bs K = \mm_B \bs J \ne \emptyset$.
\end{proof}

\begin {proof}   [Proof of Theorems~\ref{maintheorem}  and~\ref{one-d-gen}]      
As in the proof of Theorem~\ref {generalspace}, let $J$ be a Cantor set in $K$, let $B$ be a uniform algebra on $J$ such that $\mm_B \bs J$ is nonempty but $\mm_B$ contains no analytic discs, and let $A=\{ f \in C(K) : f | J \in B \}$. Note that by Theorem~\ref{Cantorset}, the uniform
algebra $ B$ can be chosen so as to be generated by three functions $f_1, f_2, f_3$.

Extend each of $f_1, f_2, f_3$ to continuous complex-valued functions $\tf_1, \tf_2, \tf_3$ on $K$.  Let $x_1 \cd x_n$ denote the real coordinate functions 
on $\R^n$.  Choose a continuous real-valued function $\rho$ on $K$ whose zero set is precisely $J$. 
Then the $n+4$ functions $\tf_1, \tf_2, \tf_3, \rho, \rho x_1 \cd \rho x_n$ generate the uniform algebra $A$ by \cite[Lemma~3.8]{ISW} .

Let $F:K\to \C^{n+4}$ be the mapping whose components are the functions  
$\tf_1, \tf_2, \tf_3, \rho, \rho x_1 \cd \rho x_n$, and 
let $X=F(K)$.  Then $P(X)$ is isomophic to $A$ as a uniform algebra and $\mm_{P(X)}$ can be identified with $\h X$.
Since the proof of Theorem~\ref{generalspace} shows that $\mm_A \bs K$ is nonempty while $\mm_A$ contains no analytic disc,
Theorem~\ref{maintheorem} follows.

An alternative approach which avoids the argument in the proof of Theorem~\ref{generalspace} is to apply \cite[Proposition 3.1(i)]{ISW} 
to show directly that, with $G:J\to \C^3$ given by $G(x)=\bigl(f_1(x), f_2(x),  f_3(x)\bigr)$, we get
\[ \h X \bs X = \Bigl(\widehat{G(J)} \bs G(J)\Bigr) \times \{0\}^{n+1} \]
Since $G(J)$ is the Cantor set of Corollary~\ref{Cantorset}, this also yields the result.  Details are left to the reader.

To establish Theorem~\ref{one-d-gen} it suffices to show that the uniform algebra $A$ can be taken to have dense invertibles when $K$ has topological dimension at most $1$, and, in view of Theorem~\ref{Cantorset-dense-invert} this is so by the following lemma.
\end{proof}

\begin{lemma}\label{dense-invert-lemma}
Let $K$ be a compact metrizable space of topological dimension $1$, and let $J$ be a compact subset of $K$.  Let $B$ be a uniform algebra on $J$ with dense invertibles, and let $A$ be the uniform algebra on $K$ defined by
$A=\{ f \in C(K) : f | J \in B \}$.  Then $A$ has dense invertibles.
\end{lemma}

\bpf
Let $f\in A$ and $\vep>0$ be arbitrary.  We are to show that there is an invertible element $g$ of $A$ such that $\|f-g\|_K<\vep$.

Since $B$ has dense invertibles, there is a function $h$ on $J$ with $h$ invertible in $B$ and with $\|(f|J)-h\|_J<\vep/2$.  It suffices to show that $h$ has a continuous extension $g$ to $K$ with no zeros on $K\setminus J$ and such that $\|f-g\|_K<\vep$, for then $1/g$ is a continuous function on $K$ and $(1/g)|J$ is in $B$, and hence $1/g$ is in $A$.

By the Tietze extension theorem, the function $(f|J)-h$ has a continuous extension $\alpha$ to $K$ with $\|\alpha\|_K<\vep/2$.  Let $\th=f-\alpha$.  Then $\th$ is a continuous function on $K$ such that $\th|J=h$ and $\|f-\th\|_K=\|\alpha\|_K<\vep/2$.  We must now modify $\th$ to obtain a function with no zeros.

We can write $K\setminus J$ as an increasing union of compact sets $K_1\subset K_2\subset \cdots$.  Roughly, the desired modification $g$ of $\th$ will be obtained by successively adding small functions $\ta_1, \ta_2, \ldots$ to $\th$ in such a way that the addition of the $m^{\rm th}$ function eliminates zeros on $K_m$, and the addition of all the subsequent functions does not reintroduce zeros on $K_m$.  More precisely, we argue as follows.  

As a subset of the at most 1-dimensional space $K$, each $K_n$ has dimension at most $1$ \cite[Theorem~III 1]{HW}.  Consequently each of the algebras $C(K_n)$ has dense invertibles \cite[Theorem~7]{Vas}.  Thus we can choose a function $\alpha_1\in C(K_1)$ such that $\|\alpha_1\|_{K_1}<\vep/2^2$ and $(\th|K_1)+\alpha_1$ has no zeros on $K_1$.  By the Tietze extension theorem, we can extend $\alpha_1$ to a continuous function $\ta_1$ on $K$ that is identically zero on $J$ and satisfies $\|\ta_1\|_K<\vep/2^2$.

Now suppose for the purpose of induction that we have chosen functions $\ta_1,\ldots,\ta_n\in C(K)$ such that each $\ta_j$ is identically zero on $J$, the function $\th + \ta_1 + \cdots +\ta_n$ has no zeros on $K_n$, the norm $\|\ta_j\|_K$ is bounded above by $\vep/2^{j+1}$, and in addition, for each $j=2,\ldots, n$, the norm $\|\ta_j\|_K$ is bounded above by each of the numbers 
$$\min_{K_m} (\th + \ta_1 + \cdots + \ta_m)/ 2^j\quad {\rm for\ }m=1, \ldots, j-1.$$ 
Then in the same manner as was used to obtain the function $\ta_1$, we can get a function $\ta_{n+1}\in C(K)$ such that $\ta_{n+1}$ is identically zero on $J$, the function $\th + \ta_1+\cdots + \ta_{n+1}$ has no zeros on $K_{n+1}$, and the norm of $\|\ta_{n+1}\|_K$ is bounded above by $\vep/2^{n+2}$ and by each of the numbers 
$$\min_{K_m} (\th + \ta_1 + \cdots + \ta_m)/ 2^{n+1}\quad {\rm for\ } m=1,\ldots, n. $$
Thus the induction can proceed.

Now the series $\sum_{j=1}^\infty \ta_j$ converges uniformly to a continuous function on $K$ and $\|\sum_{j=1}^\infty \ta_j\|_K<\vep/2$.  Thus setting $g=\th +\sum_{j=1}^\infty \ta_j$, we have
\[ \|f-g\|_K\leq \|f-\th\| + \Bigl\|{\textstyle\sum\limits_{j=1}^\infty} \ta_j\Bigr\| <\vep.
\]
Also $g|J=\th|J=h$.  In addition, given a positive integer $m$,
because 
\[ \Bigl\|{\textstyle\sum\limits_{j=m+1}^\infty} \ta_j\Bigr\| \leq
{\textstyle\sum\limits_{j=m+1}^\infty} \|\ta_j\|< \min_{K_m} (\th + \ta_1 + \cdots \ta_m)/ 2^m,
\]
the function $g$ has no zeros on $K_m$.  Thus
$g$ has no zeros on $K\setminus J$.  We conclude that $g$ is the required invertible element of $A$.
\epf

\begin{proof} [Proof of Corollaries~\ref{maincorollary} and~\ref{one-d-special}]
Since every compact metrizable space of topological dimension $m$ can be embedded in $\R^{2m+1}$ \cite[Theorem~V 2]{HW}, 
Corollaries~\ref{maincorollary} and~\ref{one-d-special} are immediate consequences of Theorems~\ref{maintheorem} and~\ref{one-d-gen}, respectively.
\end {proof}

\begin{proof} [Proof of Theorem~\ref{arc}]

Let $J$ be the Cantor set in $\C^3$ given by Theorem~\ref{Cantorset-dense-invert}. By Antoine's theorem \cite{Antoine} (or see \cite{Whyburn} where a more general result is given), there is an arc $I$ in $\C^3$
containing $J$.  We may choose $I$ such that the end points of $I$ lie in $J$.  Then $I \bs J$ is topologically a countable union of open intervals $I_1, I_2, \ldots$. Denote the end points of $I_n$ in $I$ by $a_n$ and $b_n$, and let $s_n$ be a point in the open interval $I_n$.  
For each $n=1, 2 \cd$ choose a continuous real-valued function $\rho_n$ on $I$ such that $\rho_n$ is 
identically zero on $I \bs I_n$, $\rho_n(s_n) = 1/2^n$, and $\rho_n$ is strictly increasing on the interval $(a_n,s_n)$ and
strictly decreasing on the interval $(s_n,b_n)$.  Then the series $\sum ^\infty_{n=1} \rho_n$ defines a continuous real-valued 
function $\rho$ on $I$ whose zero set is $J$ and each of whose other level sets is finite.

Let $X$ denote the graph of $\rho$.  Then $X$ is an arc in $\C^4$ and applying \cite[Proposition 3.1(i)]{ISW} (taking the mapping $F$ there to be the last coordinate function $z_4$) shows that
\[ \h X = X \cup ( \widehat J \times \{0\}) \]
so $ \h X \bs X = ( \widehat J \setminus J) \times \{0\}$.
Thus the polynomial hull of $X$ is nontrivial but contains no analytic discs.  Furthermore, note that because each level set of $\rho$, other than the zero set $J$, is a finite set, the Bishop antisymmetric decomposition shows that $P(X)=\{f\in C(X): f|(J\times\{0\})\in P(J\times\{0\})\}$.  Thus $P(X)$ has dense invertibles by Lemma~\ref{dense-invert-lemma}.

The proof with ``arc'' replaced by ``simple closed curve'' is essentially the same.
\end{proof}

\begin{remark}
One can also prove, along the same lines as the proof of Theorems~\ref{maintheorem} and~\ref{one-d-gen}, results about the existence of large Gleason parts and nonzero bounded point derivations in hulls of arbitrary uncountable, compact metrizable spaces of finite topological dimension similar to Theorems~\ref{ptderivation1}--\ref{ptderivation3} for Cantor sets.  This is left to the interested reader.
\end{remark}

\end{document}